\documentclass{article}

\usepackage{cite}      
\usepackage{graphicx}
\usepackage{graphics}
\usepackage{epsfig}
\usepackage{psfrag}   
\usepackage{subfigure} 
\usepackage{enumerate}
\usepackage{amsmath}
\usepackage{amssymb}
\usepackage{amsthm}
\usepackage{deleq}   
\usepackage{type1cm} 
\usepackage{float}
\usepackage{deleq}   
\usepackage{type1cm} 

\RequirePackage[left=1in,right=1in,top=1in,bottom=1in]{geometry}

\newtheorem{defnt}{Definition}
\newtheorem{remark}{Remark}
\newtheorem{thm}{Theorem}
\newtheorem{lemma}{Lemma}

\newtheorem{cor}{Corollary}

\title{A Control-Oriented Notion of Finite State Approximation}
\author{Danielle C.~Tarraf\footnote{D. C. Tarraf is with the Electrical \& Computer Engineering Department at the 
Johns Hopkins University, Baltimore, MD (dtarraf@jhu.edu).}}

\begin{document}

\markboth{Manuscript submitted to IEEE Transactions on Automatic Control}%
{Shell \MakeLowercase{\textit{et al.}}: }

\maketitle

\begin{abstract}

We consider the problem of approximating discrete-time plants with finite-valued sensors and actuators 
by deterministic finite memory systems for the purpose of certified-by-design controller synthesis.
Building on ideas from robust control,
we propose a control-oriented notion of finite state approximation for these systems, 
demonstrate its relevance to the control synthesis problem,
and discuss its key features.
\end{abstract}

\section{Introduction}

High fidelity models that accurately describe a dynamical system are often too complex for use in controller design. 
The problem of finding a lower complexity approximate model
has thus been extensively studied and continues to receive much deserved attention. 
A model complexity reduction approach should ideally provide both a lower complexity model and
a rigorous assessment of the quality of approximation,
allowing one to quantify the performance of a controller designed for the lower complexity model 
and implemented in the original system. 
The problem of approximating hybrid systems 
by simpler systems has received considerable attention recently \cite{JOUR:AlHeLP2000, JOUR:BBEFKP2007}:
In particular, {\it finite} state approximations of hybrid systems have been the object of 
intense study, due to the amenability of finite state models to control synthesis.
Two frameworks have been systematically explored:
`Qualitative models' and `simulation/bisimulation abstractions'.

`Qualitative models' refers to non-deterministic finite automata 
whose input/output behavior contains that of the original model. 
Control synthesis can be formulated as a supervisory control problem, 
addressed in the Ramadge-Wonham framework \cite{JOUR:RamWon1987, JOUR:RamWon1989}.
The results on qualitative models \cite{JOUR:Lunze1994},
qualitative reconstruction from quantized observations \cite{JOUR:RaiOYo1998}
and $l$-complete approximations \cite{JOUR:MooRai1999, JOUR:MoRaOY2002} fall in this category.
These approaches typically address output feedback problems.

`Simulation/bisimulation abstractions' collectively refers to a set
of related approaches inspired by bisimulation in concurrent processes.
These approaches ensure that the set of state trajectories of the original model is exactly matched by (bisimulation),
contained in (simulation),
matched to within some distance $\epsilon$ by (approximate bisimulation),
or contained to within some distance $\epsilon$ in (approximate simulation),
the set of state trajectories of the finite state abstraction \cite{JOUR:GirPap2007, JOUR:Tabuad2008, JOUR:TaAmAP2008}.
The performance objectives are typically formulated as constraints
on the state trajectories of the original hybrid system,
and controller synthesis is a two step procedure: 
A finite state supervisory controller is designed and subsequently refined 
to yield a certified hybrid controller for the original plant \cite{BOOK:Tabuad2009}.
These approaches typically address state feedback problems.
  
In our past research efforts, we proposed `$\rho/\mu$ gain' conditions to describe system properties, 
and presented a corresponding set of tools for verifying performance and robustness \cite{JOUR:TaMeDa2008}.
We also showed that for deterministic finite state machines,
we can systematically design feedback controllers to 
achieve specified $\rho/\mu$ gain conditions \cite{JOUR:TaMeDa2011}.
We demonstrated the use of these tools and a particular
approximation algorithm to synthesize finite state stabilizing controllers for 
switched homogeneous second order systems with binary sensors \cite{BOOKCHAPTER:TaMeDa2007, JOUR:TaMeDa2011}.
In this note,
we formalize a control-oriented notion of finite state approximation
for output feedback problems where the sensor information is coarse and actuation is finite valued.
This notion is compatible with the developed analysis and synthesis tools, 
thus contributing to the development of a new framework for finite state machine based
certified-by-design control.
While the proposed notion is inspired from robust control theory,
the class of problems considered here poses unique challenges
due to the lack of algebraic structure
(input/output signals take their values in arbitrary sets of symbols)
and the need to approximate both the dynamics and the performance objectives 
while appropriately quantifying the approximation error.

{\it Notation:}
$\mathbb{R}$, $\mathbb{Z}_+$ and $\mathbb{R}_+$ denote the reals,
non-negative integers and non-negative reals, respectively. 
Given a set $\mathcal{A}$, 
$\mathcal{A}^{\mathbb{Z}_+}$ denotes the set of all infinite sequences over $\mathcal{A}$ 
(indexed by $\mathbb{Z}_+$) 
and $2^{\mathcal{A}}$ denotes the power set of $\mathcal{A}$. 
Elements of $\mathcal{A}$ and $\mathcal{A}^{\mathbb{Z}_+}$ are denoted by $a$ 
and (boldface) $\mathbf{a}$, respectively. 
For $\mathbf{a} \in \mathcal{A}^{\mathbb{Z}_+}$, $a(i)$ denotes its $i^{th}$ term.
For $f: A \rightarrow B$, $C \subset B$, 
$f(A) = \{ b \in B | b=f(a) \textrm{   for some   } a \in A\}$
and $f^{-1}(C) = \{ a \in A | f(a) \in C \}$.

\section{Preliminaries}
\label{Sec:Preliminaries}

We briefly review some basic concepts:
Readers are referred to \cite{JOUR:TaMeDa2008} for a more detailed treatment.
A discrete-time signal is understood to be an infinite sequence over some prescribed set (or `alphabet').

\begin{defnt}
\label{def:system}
A discrete-time system $S$ is a set of pairs of signals, $S \subset \mathcal{U}^{\mathbb{Z}_+} \times \mathcal{Y}^{\mathbb{Z}_+}$,
where $\mathcal{U}$ and $\mathcal{Y}$ are given alphabets.
\end{defnt}

A discrete-time system is thus a process characterized by its feasible signals set. 
This view of systems can be considered an extension of the graph theoretic approach \cite{JOUR:GeoSmi1997}
to include the finite alphabet setting.
It also shares some similarities with Willems' behavioral approach \cite{JOUR:Willem2007},
although we insist on differentiating between input and output signals upfront.
In this setting, 
system properties of interest are captured by means of `integral' constraints on the feasible signals.

\begin{defnt} 
\label{def:GainStability}
Consider a system $S \subset \mathcal{U}^{\mathbb{Z}_+} \times \mathcal{Y}^{\mathbb{Z}_+}$ and let $\rho: \mathcal{U} \rightarrow \mathbb{R}$ 
and $\mu: \mathcal{Y} \rightarrow \mathbb{R}$ be given functions. $S$ is \textit{$\rho / \mu$ gain stable} if there exists a finite 
non-negative constant $\gamma$ such that 
\begin{equation}
\label{eq:gain}
\inf_{T \geq 0} \sum_{t=0}^{T} \gamma \rho (u(t)) - \mu (y(t))  > - \infty
\end{equation}
is satisfied for all $(\mathbf{u},\mathbf{y})$ in $S$.
\end{defnt}

In particular, when $\rho$ and $\mu$ are non-negative (and not identically zero), the `gain' can be defined.

\begin{defnt}
\label{def:Gain}
Consider a system $S \subset \mathcal{U}^{\mathbb{Z}_+} \times \mathcal{Y}^{\mathbb{Z}_+}$.
Assume that $S$ is $\rho / \mu$ gain stable for $\rho: \mathcal{U} \rightarrow \mathbb{R}_+$ and
$\mu: \mathcal{Y} \rightarrow \mathbb{R}_+$, and that neither function is identically zero.
The $\rho / \mu$ \textit{gain of} $S$ is the infimum of $\gamma$ such that (\ref{eq:gain}) is satisfied.
\end{defnt}
We are specifically interested in discrete-time plants with finite-valued actuators and sensors:
\begin{defnt}
\label{def:system}
A \textit{system over finite alphabets} $S$ is a discrete-time system
$S \subset (\mathcal{U} \times \mathcal{R})^{\mathbb{Z}_+} \times (\mathcal{Y} \times \mathcal{V})^{\mathbb{Z}_+}$
whose alphabets $\mathcal{U}$ and $\mathcal{Y}$ are finite.
\end{defnt}

Here, 
$\mathbf{r} \in \mathcal{R}^{\mathbb{Z}_+}$ and $\mathbf{u} \in \mathcal{U}^{\mathbb{Z}_+}$ 
represent the exogenous and control inputs to the plant, respectively,
while $\mathbf{v} \in \mathcal{V}^{\mathbb{Z}_+}$ and
$\mathbf{y} \in \mathcal{Y}^{\mathbb{Z}_+}$
represent the performance and sensor outputs of the plant, respectively.
The plant dynamics may be analog, discrete or hybrid.
Alphabets $\mathcal{R}$ and $\mathcal{V}$ may be finite, countable or infinite.
The approximate models of the plant will be drawn from a specific class of models:

\begin{defnt}
\label{def:DFM}
A deterministic finite state machine (DFM) is a discrete-time system
$S \subset \mathcal{U}^{\mathbb{Z}_+} \times \mathcal{Y}^{\mathbb{Z}_+}$
with finite alphabets $\mathcal{U}$, $\mathcal{Y}$,
whose feasible input and output signals ($\mathbf{u}$, $\mathbf{y}$) are related by
a state transition equation and an output equation:
\begin{eqnarray*}
q(t+1) & = & f (q(t), u(t)), \\
y(t) & = & g(q(t),u(t))
\end{eqnarray*}
where $t \in \mathbb{Z}_+$, $q(t) \in \mathcal{Q}$ for some finite set $\mathcal{Q}$ and functions $f: \mathcal{Q} \times \mathcal{U} \rightarrow \mathcal{Q}$ 
and $g: \mathcal{Q} \times \mathcal{U} \rightarrow \mathcal{Y}$.
\end{defnt}

Finally, we introduce the following notation for convenience:
Given a system $P \subset (\mathcal{U} \times \mathcal{R})^{\mathbb{Z}_+} \times (\mathcal{Y} \times \mathcal{V})^{\mathbb{Z}_+}$
and a choice of signals $\mathbf{u_o} \in \mathcal{U}^{\mathbb{Z}_+}$ and
$\mathbf{y_o} \in \mathcal{Y}^{\mathbb{Z}_+}$,
$P|_{\mathbf{u_o,y_o}}$ denotes the subset of feasible signals of $P$ 
whose first component is $\mathbf{u_o}$ and whose third component is $\mathbf{y_o}$.
That is
\begin{displaymath}
P |_{\mathbf{u_o},\mathbf{y_o}} = \Big\{ \Big( (\mathbf{u},\mathbf{r}),(\mathbf{y},\mathbf{v}) \Big) \in P \Big| \mathbf{u}=\mathbf{u_o} \textrm{   and   } \mathbf{y}=\mathbf{y_o} \Big\}.
\end{displaymath}
Note that $P|_{\mathbf{u_o,y_o}}$ may be an empty set for specific choices of $\mathbf{u_o}$ and $\mathbf{y_o}$.

\section{Control-Oriented Finite State Approximation}
\label{Sec:Approximation}

   \begin{figure*}[thpb]
       \centering
       \includegraphics[scale=0.5]{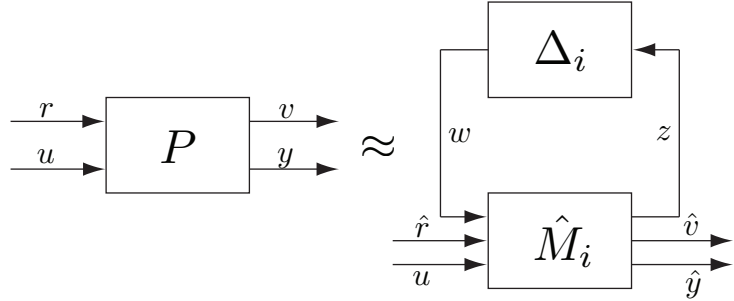}
       \caption{A finite state approximation of $P$}
       \label{Fig:Approximation}
   \end{figure*}

In this section we develop a new, 
control-oriented notion of finite state approximation for systems over finite alphabets:
We assume that the purpose of deriving a DFM approximation of a system $P$ 
over finite alphabets is to simplify the process of synthesizing a controller 
$K$ such that the closed loop system
$(P,K)$ 
is $\rho / \mu$ gain stable with $\gamma=1$ for some given $\rho$ and $\mu$.

\subsection{Proposed Notion}
\label{SSec:Approximation}

\begin{defnt}[Notion of DFM Approximation]
\label{Def:DFMApproximation}
Consider a system over finite alphabets $P \subset (\mathcal{U} \times \mathcal{R})^{\mathbb{Z}_+} \times (\mathcal{Y} \times \mathcal{V})^{\mathbb{Z}_+}$
and a desired closed loop performance objective
\begin{equation}
\label{eq:ObjectiveP}
\inf_{T \geq 0} \sum_{t=0}^{T} \rho(r(t)) - \mu(v(t)) > -\infty.
\end{equation}
for given functions
$\rho: \mathcal{R} \rightarrow \mathbb{R}$ and $\mu: \mathcal{V} \rightarrow \mathbb{R}$.
A sequence $\{\hat{M}_i\}_{i=1}^{\infty}$ of deterministic finite state machines
$\hat{M}_i \subset (\mathcal{U} \times \hat{\mathcal{R}}_i \times \mathcal{W})^{\mathbb{Z}_+} 
\times (\mathcal{Y} \times \hat{\mathcal{V}}_i \times \mathcal{Z})^{\mathbb{Z}_+}$
with $\hat{\mathcal{R}}_i \subset \mathcal{R}$ and $\hat{\mathcal{V}}_i \subset \mathcal{V}$ is a {\boldmath $\rho / \mu$} 
\textbf{approximation} of $P$
if there exists a corresponding sequence of systems $\{\Delta_i\}_{i=1}^{\infty}$,
 $\Delta_i \subset \mathcal{Z}^{\mathbb{Z}^+} \times \mathcal{W}^{\mathbb{Z}_+}$, 
and non-zero functions $\rho _{\Delta}:\mathcal{Z} \rightarrow \mathbb{R}_+$,
$\mu_{\Delta}: \mathcal{W} \rightarrow \mathbb{R}_+$, such that for every index $i$:

\begin{enumerate}[(a)]

\item There exists a surjective map $\psi_i: P \rightarrow \hat{P}_i$
satisfying
$$\psi_i \Big( P|_{\mathbf{u},\mathbf{y}} \Big) \subseteq \hat{P}_i |_{\mathbf{u},\mathbf{y}} $$
for all $(\mathbf{u},\mathbf{y}) \in \mathcal{U}^{\mathbb{Z_+}} \times \mathcal{Y}^{\mathbb{Z}_+}$,
where 
$\hat{P}_i \subset (\mathcal{U} \times \hat{\mathcal{R}}_i)^{\mathbb{Z}_+} \times (\mathcal{Y} \times \hat{\mathcal{V}}_i)^{\mathbb{Z}_+} $ 
is the feedback interconnection of $\hat{M}_i$ and $\Delta_i$ as shown in Figure \ref{Fig:Approximation}.

\item For every feasible signal
$((\mathbf{u},\mathbf{r}),(\mathbf{y},\mathbf{v})) \in P$, we have
\begin{eqnarray}
\label{eq:Objectivenounds}
\rho(r(t)) - & \mu(v(t)) \geq \rho(\hat{r}_{i+1}(t)) - \mu(\hat{v}_{i+1}(t)) 
\geq & \rho(\hat{r}_i(t)) - \mu(\hat{v}_i(t)),  
\end{eqnarray}
for all $t \in \mathbb{Z}_+$, where 
$ ((\mathbf{u},\mathbf{\hat{r}_i}),(\mathbf{\hat{y}_i},\mathbf{\hat{v}_i})) = \psi_i \Big( ( (\mathbf{u},\mathbf{r}),(\mathbf{y},\mathbf{v}))  \Big)$
and 
$((\mathbf{u},\mathbf{\hat{r}_{i+1}}),(\mathbf{\hat{y}_{i+1}},\mathbf{\hat{v}_{i+1}})) = \psi_{i+1} \Big( ((\mathbf{u},\mathbf{r}),(\mathbf{y},\mathbf{v}))  \Big)$.

\item $\Delta_i$ is $\rho_{\Delta} / \mu_{\Delta}$ gain stable,
and moreover, the corresponding $\rho_{\Delta} / \mu_{\Delta}$ gains satisfy $\gamma_{i} \geq \gamma_{i+1}$. 

\end{enumerate}  
\end{defnt}

\smallskip
\begin{remark}
Note that in this setup, the dynamics of plant $P$ as well as alphabet sets $\mathcal{U}$ and $\mathcal{Y}$
are given (in practice, defined by the system and hardware).
We also have no influence over the exogenous input $r$. 
In contrast, in addition to choosing $\hat{M}_i$ and $\Delta_i$,
we are typically free to define the performance output $v$
(which can be an arbitrary function of the state of $P$ and its inputs) 
to suit our purposes.
We are likewise free to pick functions $\rho$, $\mu$,
and non-negative functions $\rho_{\Delta}$, $\mu_{\Delta}$ 
to suit our purposes.
The proposed  notion of approximation thus provides some margin of flexibility,
and the details of the problem (both the dynamics and the desired performance) 
largely influence our choice of  signals, gain conditions, and approximate models.

\end{remark}

\subsection{Relevance to Verifably Correct Control Synthesis}
\label{SSec:Significance}

We begin by establishing several facts that will help demonstrate the relevance of the proposed notion 
of approximation to the problem of certified-by-design controller synthesis.

\begin{lemma}
\label{Lemma:EquivalenceClasses}
Consider a plant $P$ and a $\rho/\mu$ approximation $\{\hat{M}_i\}$ as in Definition \ref{Def:DFMApproximation}.
The (non-empty) sets $P|_{\mathbf{u},\mathbf{y}}$, 
$(\mathbf{u},\mathbf{y}) \in \mathcal{U}^{\mathbb{Z}_+} \times \mathcal{Y}^{\mathbb{Z}_+}$,
partition $P$ into equivalence classes.
For every index $i$, the (non-empty) sets $\hat{P}_i|_{\mathbf{u},\mathbf{y}}$, 
$(\mathbf{u},\mathbf{y}) \in \mathcal{U}^{\mathbb{Z}_+} \times \mathcal{Y}^{\mathbb{Z}_+}$,
partition $\hat{P}_i$ into equivalence classes.
\end{lemma}

\begin{proof}
It immediately follows from the definition that $P|_{\mathbf{u_1},\mathbf{y_1}} \cap P|_{\mathbf{u_2},\mathbf{y_2}}= \emptyset$ 
whenever $(\mathbf{u_1},\mathbf{y_1}) \neq (\mathbf{u_2},\mathbf{y_2})$.
It also follows from the definition that every $((\mathbf{u},\mathbf{r}),(\mathbf{y},\mathbf{v}))$ in $P$ belongs 
to some $P|_{\mathbf{u},\mathbf{y}}$, hence $\displaystyle \bigcup_{\mathbf{u},\mathbf{y}} P|_{\mathbf{u},\mathbf{y}} = P$.
The proof for each $\hat{P}_i$ is similar and is thus omitted for brevity.
\end{proof}

\begin{lemma}
\label{Lemma:Preimage}
Consider a plant $P$ and a $\rho/\mu$ approximation $\{\hat{M}_i\}$ as in Definition \ref{Def:DFMApproximation}.
For every index $i$, $(\mathbf{u},\mathbf{y}) \in \mathcal{U}^{\mathbb{Z}_+} \times \mathcal{Y}^{\mathbb{Z}_+}$, 
we have $\psi_i \Big( P|_{\mathbf{u},\mathbf{y}}\Big) =  \hat{P}_i|_{\mathbf{u},\mathbf{y}}$.
\end{lemma}

\begin{proof} 
By condition (a) of Definition \ref{Def:DFMApproximation}, 
for each $i$
there exists a $\psi_i : P \rightarrow \hat{P}_i$
with $\psi_i \Big( P|_{\mathbf{u},\mathbf{y}}\Big) \subseteq \hat{P}_i |_{\mathbf{u},\mathbf{y}}$
for all $(\mathbf{u},\mathbf{y}) \in \mathcal{U}^{\mathbb{Z}_+} \times \mathcal{Y}^{\mathbb{Z}_+}$.
What remains is to show equality.
Fix index $i$.
For a given choice of $(\mathbf{u},\mathbf{y}) \in \mathcal{U}^{\mathbb{Z}_+} \times \mathcal{Y}^{\mathbb{Z}_+}$:
If $\hat{P}_i|_{\mathbf{u},\mathbf{y}} = \emptyset$, 
we have $\psi_i \Big( P|_{\mathbf{u},\mathbf{y}} \Big) \subseteq \hat{P}_i |_{\mathbf{u},\mathbf{y}} = \emptyset$,
and equality holds.
Otherwise,
assume there exists an 
$x \in \hat{P}_i |_{\mathbf{u},\mathbf{y}}$ such that 
$x \notin \psi_i \Big( P|_{\mathbf{u},\mathbf{y}}\Big)$.
Since $\psi_i$ is surjective, 
$x \in \psi_i \Big( P|_{\mathbf{u}_{1},\mathbf{y}_{1}} \Big)$ for some 
$(\mathbf{u}_{1},\mathbf{y}_{1}) \neq (\mathbf{u},\mathbf{y})$.
We then have $x \in \hat{P}_i|_{\mathbf{u},\mathbf{y}} \cap \hat{P}_i|_{\mathbf{u}_{1},\mathbf{y}_{1}}$,
leading to a contradiction by Lemma \ref{Lemma:EquivalenceClasses}.
Thus, such an $x$ cannot exist, and equality holds.
Finally, note that the proof is independent of the choice of index $i$. 
\end{proof}

\begin{cor}
\label{Cor:Emptyset}
Consider a plant $P$ and a $\rho/\mu$ approximation $\{\hat{M}_i\}$ as in Definition \ref{Def:DFMApproximation}.
For every index $i$, $(\mathbf{u},\mathbf{y}) \in \mathcal{U}^{\mathbb{Z}_+} \times \mathcal{Y}^{\mathbb{Z}_+}$, 
we have $P|_{\mathbf{u},\mathbf{y}} = \emptyset$ iff $\hat{P}_i|_{\mathbf{u},\mathbf{y}} = \emptyset$.
\end{cor}

\begin{proof}
For any index $i$, we have
$ \displaystyle \hat{P}_i | _{\mathbf{u},\mathbf{y}} = \emptyset 
\Leftrightarrow 
\psi_i \Big( P |_{\mathbf{u},\mathbf{y}}\Big) = \emptyset
\Leftrightarrow
P |_{\mathbf{u},\mathbf{y}} = \emptyset$
where the first equivalence follows from Lemma \ref{Lemma:Preimage}.
\end{proof}

As a consequence of these simple facts,
if we were to partition each of $P$ and $\hat{P}_i$ into equivalence classes of feasible signals having identical first and third components 
(corresponding to control inputs and sensor outputs),
the existence of a surjective map $\psi_i$ satisfying condition (a) of Definition \ref{Def:DFMApproximation}
effectively establishes a  1-1 correspondence between the equivalence classes of $P$ and $\hat{P}_i$.
Moreover, it  follows from condition (b) of Definition \ref{Def:DFMApproximation} that
if all signals in a given equivalence class of $\hat{P}_i$ satisfy a $\rho/\mu$ gain stability condition,
then so do all the signals of the corresponding equivalence class of $P$.
This is formalized and proved in the following statements.

\begin{cor}
\label{Cor:ClassesBijection}
Consider a plant $P$ and a $\rho/\mu$ approximation $\{\hat{M}_i\}$ as in Definition \ref{Def:DFMApproximation}.
For every index $i$, there exists a bijection between the equivalence classes 
$\{ P|_{\mathbf{u},\mathbf{y}}\}$ of $P$ and $\{ \hat{P}_i|_{\mathbf{u},\mathbf{y}} \}$ of $\hat{P}_i$.
\end{cor}

\begin{proof}
For every index $i$, consider the map 
$\Psi_i: \{ P|_{\mathbf{u},\mathbf{y}}\} \rightarrow \{ \hat{P}_i|_{\mathbf{u},\mathbf{y}} \}$ 
defined by $\Psi_i(P|_{\mathbf{u},\mathbf{y}}) = \psi_i ( P|_{\mathbf{u},\mathbf{y}})$.
Note that the choice of codomain for $\Psi_i$ is valid by Lemma \ref{Lemma:Preimage}.
$\Psi_i$ is injective: 
\begin{eqnarray*}
\Psi_i(P|_{\mathbf{u_1},\mathbf{y_1}}) =\Psi_i(P|_{\mathbf{u_2},\mathbf{y_2}}) 
& \Rightarrow &  \hat{P}_i|_{\mathbf{u_1},\mathbf{y_1}} =  \hat{P}_i|_{\mathbf{u_2},\mathbf{y_2}}\\
& \Rightarrow & (\mathbf{u_1},\mathbf{y_1})=(\mathbf{u_2},\mathbf{y_2}) \\
& \Rightarrow & P|_{\mathbf{u_1},\mathbf{y_1}} = P|_{\mathbf{u_2},\mathbf{y_2}}
\end{eqnarray*}
with the first implication following from Lemma \ref{Lemma:Preimage} and the 
second implication following from Corollary \ref{Cor:Emptyset}.
Indeed, we can exclude the possibility that 
$\hat{P}_i|_{\mathbf{u_1},\mathbf{y_1}} =  \hat{P}_i|_{\mathbf{u_2},\mathbf{y_2}} = \emptyset$ 
in the second implication as that would imply (by Corollary \ref{Cor:Emptyset}) that
$P|_{\mathbf{u_1},\mathbf{y_1}} =  P|_{\mathbf{u_2},\mathbf{y_2}} = \emptyset$ 
which is false by assumption.
$\Psi_i$ is surjective: For every $\hat{P}_i|_{\mathbf{u},\mathbf{y}} \neq \emptyset$,
there exists $P|_{\mathbf{u},\mathbf{y}} \neq \emptyset$ (by Corollary \ref{Cor:Emptyset})
such that $\Psi_i(P|_{\mathbf{u},\mathbf{y}}) = \hat{P}_i|_{\mathbf{u},\mathbf{y}}$.
Therefore, $\Psi_i$ is bijective.
\end{proof}

\begin{lemma}
\label{Lemma:SatisfyingIC}
Consider a plant $P$ and a $\rho/\mu$ approximation $\{\hat{M}_i\}$ as in Definition \ref{Def:DFMApproximation}.
For any choice of index $i$ and of
$(\mathbf{u},\mathbf{y}) \in \mathcal{U}^{\mathbb{Z}_+} \times \mathcal{Y}^{\mathbb{Z}_+}$,
if every $((\mathbf{u},\mathbf{\hat{r}}),(\mathbf{y},\mathbf{\hat{v}})) \in \hat{P}_i|_{\mathbf{u},\mathbf{y}}$ satisfies 
\begin{equation}
\label{eq:ObjectiveS}
\inf_{T \geq 0} \sum_{t=0}^{T} \rho(\hat{r}(t)) - \mu(\hat{v}(t)) > -\infty
\end{equation}
then every $((\mathbf{u},\mathbf{r}),(\mathbf{y},\mathbf{v})) \in P|_{\mathbf{u},\mathbf{y}}$ satisfies (\ref{eq:ObjectiveP}).
\end{lemma}

\begin{proof}
Fix $i$ and consider any $(\mathbf{u},\mathbf{y}) \in \mathcal{U}^{\mathbb{Z}_+} \times \mathcal{Y}^{\mathbb{Z}_+}$.
If $\hat{P}_i|_{\mathbf{u,\mathbf{y}}} = \emptyset$, 
then $P|_{\mathbf{u},\mathbf{y}} = \emptyset$ by Corollary \ref{Cor:Emptyset} 
and the statement holds vacuously.
Now suppose that $\hat{P}_i|_{\mathbf{u},\mathbf{y}} \neq \emptyset$ and every 
$((\mathbf{u},\mathbf{\hat{r}}),(\mathbf{y},\mathbf{\hat{v}})) \in \hat{P}_i|_{\mathbf{u},\mathbf{y}}$ 
satisfies (\ref{eq:ObjectiveS}).
Pick any $((\mathbf{u},\mathbf{r}),(\mathbf{y},\mathbf{v})) \in P|_{\mathbf{u},\mathbf{y}}$ and consider
its image $\psi_i \Big( ((\mathbf{u},\mathbf{r}),(\mathbf{y},\mathbf{v}))  \Big) = ((\mathbf{u},\mathbf{\hat{r}}),(\mathbf{y},\mathbf{\hat{v}}))$.
By condition (b) of Definition \ref{Def:DFMApproximation}, we have
\begin{align*}
 \rho(r(t)) - \mu(v(t)) \geq \rho(\hat{r}(t)) - \mu(\hat{v}(t)), \textrm{    } \forall t
& \Rightarrow \sum_{t=0}^{T} \rho(r(t)) - \mu(v(t)) \geq \sum_{t=0}^{T} \rho(\hat{r}(t)) - \mu(\hat{v}(t)), \textrm{    } \forall T \\
& \Rightarrow \sum_{t=0}^{T} \rho(r(t)) - \mu(v(t)) \geq \inf_{T \geq 0} \sum_{t=0}^{T} \rho(\hat{r}(t)) - \mu(\hat{v}(t)), \textrm{    } \forall T \\
& \Rightarrow \inf_{T \geq 0} \sum_{t=0}^{T} \rho(r(t)) - \mu(v(t)) \geq \inf_{T \geq 0} \sum_{t=0}^{T} \rho(\hat{r}(t)) - \mu(\hat{v}(t))
\end{align*}

Thus if every element of $\hat{P}_{i}|_{\mathbf{u},\mathbf{y}}$ satisfies (\ref{eq:ObjectiveS}),
then every element of $P|_{\mathbf{u},\mathbf{y}}$ satisfies (\ref{eq:ObjectiveP}).
\end{proof}

We are now ready to turn our attention to the problem of control synthesis.

\begin{thm}
\label{Thm:Synthesis1}
Consider a plant $P$ and a $\rho/\mu$ approximation $\{\hat{M}_i\}$ as in Definition \ref{Def:DFMApproximation}.
Let $K \subset \mathcal{Y}^{\mathbb{Z}_+} \times \mathcal{U}^{\mathbb{Z}_+}$ be such that the 
feedback interconnection 
$(\hat{P}_i, K) \subset \hat{\mathcal{R}}_i^{\mathbb{Z}_+} \times \hat{\mathcal{V}}_i^{\mathbb{Z}_+}$
satisfies (\ref{eq:ObjectiveS}) for some index $i$.
Then the feedback interconnection 
$(P, K) \subset \mathcal{R}^{\mathbb{Z}_+} \times \mathcal{V}^{\mathbb{Z}_+}$
satisfies (\ref{eq:ObjectiveP}).
\end{thm}

\begin{proof}
Let 
\begin{eqnarray*}
P|_{K} & = & \Big\{ \Big( (\mathbf{u},\mathbf{r}),(\mathbf{y},\mathbf{v}) \Big) \in P \Big| (\mathbf{y},\mathbf{u}) \in K \Big\}, \\
\hat{P}_i|_{K} & = & \Big\{ \Big( (\mathbf{u},\mathbf{\hat{r}}),(\mathbf{\hat{y}},\mathbf{\hat{v}}) \Big) \in \hat{P}_i 
\Big| (\mathbf{\hat{y}},\mathbf{u}) \in K \Big\}.
\end{eqnarray*}
Note that the closed loop systems $(P,K)$ and $(\hat{P}_i,K)$ are simply the projections of
$P|_{K}$ and $\hat{P}_i|_{K}$, respectively, along the second and fourth components:
\begin{eqnarray*}
(P,K) =  \Big\{ (\mathbf{r},\mathbf{v}) \in \mathcal{R}^{\mathbb{Z}_+} \times \mathcal{V}^{\mathbb{Z}_+} \Big| 
((\mathbf{u},\mathbf{r}),(\mathbf{y},\mathbf{v})) \in P|_{K} \textrm{  for some  }  (\mathbf{u},\mathbf{y}) \in \mathcal{U}^{\mathbb{Z}_+} \times  \mathcal{Y}^{\mathbb{Z}_+} 
\Big\},
\end{eqnarray*}
\begin{eqnarray*}
(\hat{P}_i,K) = \Big\{ (\mathbf{\hat{r}},\mathbf{\hat{v}}) \in \hat{\mathcal{R}}_i^{\mathbb{Z}_+} \times \hat{\mathcal{V}}_i^{\mathbb{Z}_+} \Big|  
((\mathbf{u},\mathbf{\hat{r}}),(\mathbf{\hat{y}},\mathbf{\hat{v}})) \in \hat{P}_i|_{K} 
\textrm{  for some  }
(\mathbf{u},\mathbf{\hat{y}}) \in \mathcal{U}^{\mathbb{Z}_+} \times  \mathcal{Y}^{\mathbb{Z}_+} \Big\}.
\end{eqnarray*}
Also note that by definition, every $(\mathbf{r},\mathbf{v})$ in $(P,K)$ satisfies (\ref{eq:ObjectiveP}) if and only if
every $((\mathbf{u},\mathbf{r}),(\mathbf{y},\mathbf{v}))$ in $P|_{K}$ satisfies (\ref{eq:ObjectiveP}).
Likewise, every $(\mathbf{\hat{r}},\mathbf{\hat{v}})$ in $(\hat{P}_i,K)$ satisfies (\ref{eq:ObjectiveS}) if
and only if
every $((\mathbf{u},\mathbf{\hat{r}}),(\mathbf{\hat{y}},\mathbf{\hat{v}}))$ in $\hat{P}_i|_{K}$ satisfies (\ref{eq:ObjectiveP}).
Now suppose that for some index $i$, $(\hat{P}_i, K)$ satisfies (\ref{eq:ObjectiveS}).
Thus for every $(\mathbf{y},\mathbf{u}) \in K$,
all the elements of $\hat{P}_i|_{\mathbf{u},\mathbf{y}}$ satisfy (\ref{eq:ObjectiveS}),
and it follows from Lemma \ref{Lemma:SatisfyingIC} that all the elements of $P|_{\mathbf{u},\mathbf{y}}$ satisfy (\ref{eq:ObjectiveP}).
Hence every element of $P|_{K}$ also satisfies (\ref{eq:ObjectiveP}), and so does $(P,K)$.
\end{proof}

Theorem \ref{Thm:Synthesis1} implies that the original problem of designing a controller $K$ for the plant $P$ to meet performance objective (\ref{eq:ObjectiveP})
can be substituted by the problem of designing a controller $K$ for some $\hat{P}_i$ to meet an auxiliary performance objective (\ref{eq:ObjectiveS}),
since any feedback controller that allows us to meet the closed loop specifications of the latter problem 
also allows us to meet the closed loop specifications of the former problem.
Of course, the problem of finding a controller $K$ such that the feedback interconnection $(\hat{P}_i, K)$ satisfies (\ref{eq:ObjectiveS}) is a difficult problem in general,
since $\Delta_i$ can be an arbitrarily complex system.
However, a simpler problem can be posed by 
utilizing the available characterization of the approximation error $\Delta_i$ in terms of $\rho_{\Delta} / \mu_{\Delta}$ 
gain stability with gain $\gamma_i$.
Similar to what is done in the classical robust control setting, the idea is to design $K$ such that the 
interconnection of $\hat{M}_i$, $K$ and {\it any} $\Delta$ in the class $\mathbf{\Delta}_i$
\begin{equation*}
\mathbf{\Delta}_i =  \{  \Delta \subset \mathcal{Z}^{\mathbb{Z}_+} \times \mathcal{W}^{\mathbb{Z}_+}| 
 \inf_{T \geq 0} \sum_{t=0}^{T} \gamma_i \rho_{\Delta}(z(t)) - \mu_{\Delta}(w(t)) > -\infty 
\mbox{  holds  } \forall (\mathbf{z},\mathbf{w}) \in \Delta \} 
\end{equation*}
satisfies the auxiliary performance objective (\ref{eq:ObjectiveS}).
This synthesis problem can be elegantly formulated using the `Small Gain Theorem'
proposed in \cite{JOUR:TaMeDa2008}.

   \begin{figure}[thpb]
      \begin{center}
      \includegraphics[scale=0.5]{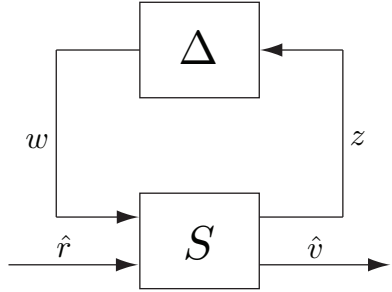}
      \caption{Setup for the `Small Gain' Theorem.}
      \label{fig:SGT}
      \end{center}
   \end{figure}

\begin{thm} [{\it Small Gain Theorem} - Adapted from \cite{JOUR:TaMeDa2008}]
\label{Thm:SGT}
Consider the feedback interconnection of two systems $S$ and $\Delta$ as in Figure \ref{fig:SGT}.
If $S$ satisfies
\begin{equation}
\label{eq:SGT1}
\inf_{T \geq 0} \sum_{t=0}^{T} \rho_{S}(\hat{r}(t),w(t)) - \mu_{S}(\hat{v}(t),z(t)) > -\infty
\end{equation}
for some $\rho_{S} : \hat{\mathcal{R}} \times \mathcal{W} \rightarrow \mathbb{R}$, 
$\mu_{S} : \hat{\mathcal{V}} \times \mathcal{Z} \rightarrow \mathbb{R}$
($\hat{\mathcal{R}}$, $\mathcal{W}$, $\hat{\mathcal{V}}$ and $\mathcal{Z}$ are finite alphabets), 
and $\Delta$ satisfies
\begin{equation} 
\label{eq:GainofDelta}
\inf_{T \geq 0} \sum_{t=0}^{T} \gamma_{\Delta} \rho_{\Delta}(z(t)) - \mu_{\Delta}(w(t)) > -\infty
\end{equation}
for some scalar $\gamma_{\Delta}$, $\rho_{\Delta}: \mathcal{Z} \rightarrow \mathbb{R}$, 
$\mu_{\Delta} : \mathcal{W} \rightarrow \mathbb{R}$,
then $(S,\Delta)$ satisfies (\ref{eq:ObjectiveS})
for $\rho: \hat{\mathcal{R}} \rightarrow \mathbb{R}$, 
$\mu: \hat{\mathcal{V}} \rightarrow \mathbb{R}$ defined by
\begin{eqnarray*}
\rho(\hat{r}) & = & \max_{w \in \mathcal{W}} \{ \rho_{S}(\hat{r},w) - \tau \mu_{\Delta}(w) \}, \\
\mu(\hat{v}) & = & \min_{z \in \mathcal{Z}} \{ \mu_{S}(\hat{v},z) - \tau \gamma_{\Delta} \rho_{\Delta}(z) \}
\end{eqnarray*}
for any $\tau>0$. \hspace{\stretch{1}} $\Box$
\end{thm} 

Interpreting Theorem \ref{Thm:SGT} where ``$S$" represents the feedback interconnection of 
$\hat{M}_i$ and $K$ and where ``$\Delta$" represents the corresponding approximation error $\Delta_i$, 
we can formulate the following:

\begin{thm}
\label{Thm:Synthesis2}
Consider a plant $P$ and a $\rho/\mu$ approximation $\{\hat{M}_i\}$ as in Definition \ref{Def:DFMApproximation}.
If for some index $i$, there exists a controller $K \subset \mathcal{Y}^{\mathbb{Z}_+} \times \mathcal{U}^{\mathbb{Z}_+}$
such that the feedback interconnection $(\hat{M}_i,K) \subset (\hat{\mathcal{R}_i} \times \mathcal{W})^{\mathbb{Z}_+} \times (\hat{\mathcal{V}_i} \times \mathcal{Z})^{\mathbb{Z}_+}$ satisfies
\begin{equation}
\label{eq:PerformanceMP}
\inf_{T \geq 0} \sum_{t=0}^{T} \rho(\hat{r}(t)) + \tau \mu_{\Delta} (w(t)) - \mu(\hat{v}(t)) - \tau \gamma_i \rho_{\Delta}(z(t)) > -\infty
\end{equation}
for some $\tau > 0$, then the feedback interconnection 
$(\hat{P}_i, K) \subset \hat{\mathcal{R}}_i ^{\mathbb{Z}_+} \times \hat{\mathcal{V}}_i^{\mathbb{Z}_+}$
satisfies (\ref{eq:ObjectiveS}).
\end{thm}

\begin{proof}
Letting $S=(\hat{M}_i,K)$,
$\Delta = \Delta_i$,
$\rho_{S} (\hat{r},w) = \rho(\hat{r}) + \tau \mu_{\Delta} (w)$,
$\mu_{S} (\hat{v},z) =  \mu(\hat{v}) + \tau \gamma_i \rho_{\Delta}(z)$,
and $\gamma_\Delta = \gamma_i$,
we have by Theorem \ref{Thm:SGT} that 
the interconnection of $K$, $\hat{M}_i$ and $\Delta_i$ satisfies (\ref{eq:ObjectiveS}).
Equivalently, the feedback interconnection of $(\hat{P}_i, K)$ satisfies (\ref{eq:ObjectiveS}).
\end{proof}

The problem of designing a controller $K$ for a DFM $\hat{M}_i$ so that the closed loop system satisfies a gain condition
(such as (\ref{eq:PerformanceMP})) can be systematically addressed by solving a 
corresponding discrete minimax problem.
Interested readers are referred to \cite{JOUR:TaMeDa2011} for the details of the approach.

Intuitively, the availability of such finite approximations
allows one to successively replace the original
synthesis problem by two problems:
The first (Theorem \ref{Thm:Synthesis1})
allows one to approximate the performance objectives when the exogenous input
and performance output of the plant are not finite valued.
The second (Theorem \ref{Thm:Synthesis2})
allows one to simplify the synthesis problem at the expense of additional conservatism 
by introducing a set based description of the approximate model.
In practice, exact computation of  $\gamma_i$ may be computationally prohibitive if not impossible.
Gain {\it bounds} are typically used,
leading to a hierarchy of synthesis problems and controllers.

\begin{thm}
\label{Prop:gainBoundHierarchy}
Consider a plant $P$ and a $\rho/\mu$ approximation $\{\hat{M}_i\}$ as in Definition \ref{Def:DFMApproximation}.
For each approximate model $\hat{M}_i$ and corresponding approximation error $\Delta_i$ with gain $\gamma_i$,
let $\displaystyle \{ \tilde{\gamma}_i^j \}_{j=1}^{\infty}$ be a sequence of gain bounds satisfying 
$\tilde{\gamma}_i^j \geq \tilde{\gamma}_i^{j+1} \geq \gamma_i$.
Let $K_j \subset \mathcal{Y}^{\mathbb{Z}_+} \times \mathcal{U}^{\mathbb{Z}_+}$,
be such that the feedback interconnection 
$(\hat{M}_i,K_j) \subset (\hat{\mathcal{R}}_i \times \mathcal{W})^{\mathbb{Z}_+} \times (\hat{\mathcal{V}}_i \times \mathcal{Z})^{\mathbb{Z}_+}$ satisfies
\begin{displaymath}
\inf_{T \geq 0} \sum_{t=0}^{T} \rho(\hat{r}(t)) + \tau \mu_{\Delta} (w(t)) - 
\mu(\hat{v}(t)) - \tau \tilde{\gamma}_i^j \rho_{\Delta}(z(t)) > -\infty
\end{displaymath}
for some $\tau>0$.
Then:
\begin{enumerate}[(a)]

\item For every $k > j$,
$(\hat{M}_i,K_j) \subset (\hat{\mathcal{R}}_i \times \mathcal{W})^{\mathbb{Z}_+} \times (\hat{\mathcal{V}}_i \times \mathcal{Z})^{\mathbb{Z}_+}$ satisfies
\begin{displaymath}
\inf_{T \geq 0} \sum_{t=0}^{T} \rho(\hat{r}(t)) + \tau \mu_{\Delta} (w(t)) - 
\mu(\hat{v}(t)) - \tau \tilde{\gamma}_i^k \rho_{\Delta}(z(t)) > -\infty
\end{displaymath}

\item $(\hat{P}_i, K_j) \subset \hat{\mathcal{R}}_i ^{\mathbb{Z}_+} \times \hat{\mathcal{V}}_i^{\mathbb{Z}_+}$
satisfies (\ref{eq:ObjectiveS}).
\end{enumerate}
\end{thm}

\begin{proof}
The proof of  statement (a) follows from the fact that $\tilde{\gamma}_i^j \geq \tilde{\gamma}_i^k$ for $k >  j$.
The proof of statement (b) follows from $\tilde{\gamma}_i^j \geq \gamma_i$ and Theorem \ref{Thm:Synthesis2}.
\end{proof}

We conclude with a final observation: 

\begin{thm} 
Consider a plant $P$ and a $\rho/\mu$ approximation $\{\hat{M}_i\}$ as in Definition \ref{Def:DFMApproximation}.
Suppose that for some index $i^*$, there exists a time $T^*$ such that 
\begin{equation}
\label{eq:StrengthenedObjectivenounds}
\rho((r(t)) - \mu(v(t)) = \rho(\hat{r}(t)) - \mu(\hat{v}(t)), \textrm{    } \forall t \geq T^*
\end{equation}
for every $((\mathbf{u},\mathbf{r}),(\mathbf{y},\mathbf{v})) \in P$,
$((\mathbf{u},\mathbf{\hat{r}}),(\mathbf{\hat{y}},\mathbf{\hat{v}})) = \psi_{i^*} \Big( ( (\mathbf{u},\mathbf{r}),(\mathbf{y},\mathbf{v}))  \Big)$.
Then, for any $K \subset \mathcal{Y}^{\mathbb{Z}_+} \times \mathcal{U}^{\mathbb{Z}_+}$,
the interconnection 
$(\hat{P}_{i^*},K) \subset \hat{\mathcal{R}}_{i^*} \times \hat{\mathcal{V}}_{i^*}$ satisfies (\ref{eq:ObjectiveS}) iff the interconnection
$(P,K) \subset \mathcal{R}^{\mathbb{Z}_+} \times \mathcal{V}^{\mathbb{Z}_+}$ satisfies (\ref{eq:ObjectiveP}).
\end{thm}

\begin{proof}
Necessity follows from Theorem \ref{Thm:Synthesis1}.
To prove sufficiency,
suppose that $(P,K)$ satisfies (\ref{eq:ObjectiveP}).
Equivalently (using the notation introduced in the proof of Theorem \ref{Thm:Synthesis1}),
every $((\mathbf{u},\mathbf{r}),(\mathbf{y},\mathbf{v})) \in P|_{K}$ satisfies (\ref{eq:ObjectiveP}).
Noting that 
$$\displaystyle P|_K= \bigcup_{(\mathbf{y},\mathbf{u}) \in K} P|_{\mathbf{u,\mathbf{y}}},$$
we can equivalently rewrite this as $P|_{\mathbf{u},\mathbf{y}}$ satisfies (\ref{eq:ObjectiveP})
for all $(\mathbf{y},\mathbf{u}) \in K$.
Now pick any $(\mathbf{y},\mathbf{u}) \in K$:
For any $((\mathbf{u},\mathbf{\hat{r}}),(\mathbf{y},\mathbf{\hat{v}})) \in \hat{P}_{i^*}|_{\mathbf{u},\mathbf{y}}$,
it follows from Lemma \ref{Lemma:Preimage} that there exists a 
$((\mathbf{u},\mathbf{r}),(\mathbf{y},\mathbf{v})) \in P|_{\mathbf{u},\mathbf{y}}$
such that 
$ \psi_{i^*} \Big( ( (\mathbf{u},\mathbf{r}),(\mathbf{y},\mathbf{v}))  \Big)= ((\mathbf{u},\mathbf{\hat{r}}),(\mathbf{y},\mathbf{\hat{v}})) $.
For $T > T^*$, we can write
\begin{eqnarray*}
 \sum_{t=0}^{T} \rho(\hat{r}(t)) - \mu(\hat{v}(t))
& =& \sum_{t=0}^{T^*} \rho(\hat{r}(t)) - \mu(\hat{v}(t)) + \sum_{t=T^*}^{T} \rho(\hat{r}(t)) - \mu(\hat{v}(t)) \\
& = & \sum_{t=0}^{T^*} \rho(\hat{r}(t)) - \mu(\hat{v}(t)) + \sum_{t=T^*}^{T} \rho((r(t)) - \mu(v(t)) \\
& = & C + \sum_{t=0}^{T} \rho((r(t)) - \mu(v(t)) \\
\end{eqnarray*}
where $C= \sum_{t=0}^{T^*} \rho(\hat{r}(t)) - \mu(\hat{v}(t))  - \sum_{t=0}^{T^*} \rho(r(t)) - \mu(v(t)) $.
We thus conclude that $((\mathbf{u},\mathbf{\hat{r}}),(\mathbf{y},\mathbf{\hat{v}}))$
satisfies (\ref{eq:ObjectiveS}).
The argument is completed by noting that the choice of  $(\mathbf{y},\mathbf{u}) \in K$ and 
$((\mathbf{u},\mathbf{\hat{r}}),(\mathbf{y},\mathbf{\hat{v}})) \in \hat{P}_i|_{\mathbf{u},\mathbf{y}}$
were arbitrary.

It follows from (\ref{eq:StrengthenedObjectivenounds}) and Lemma \ref{Lemma:Preimage},
using an argument similar to that made in Theorem \ref{Thm:Synthesis1} (omitted here for brevity),
that $(\hat{P}_{i^*},K)$ satisfies (\ref{eq:ObjectiveS}).
\end{proof}

\smallskip
\begin{remark}
In practice, an iterative procedure is used, 
whereby the first component of the $\rho / \mu$ approximation sequence is constructed and 
control synthesis is attempted.
If synthesis is succesful, we are done; 
Otherwise, the next component of the sequence is constructed and 
our attempt at control synthesis is repeated.
\end{remark}

\subsection{Illustrative Example}
\label{SSec:Example}

 \begin{figure}[thpb]
 \centering
       \includegraphics[scale=0.52]{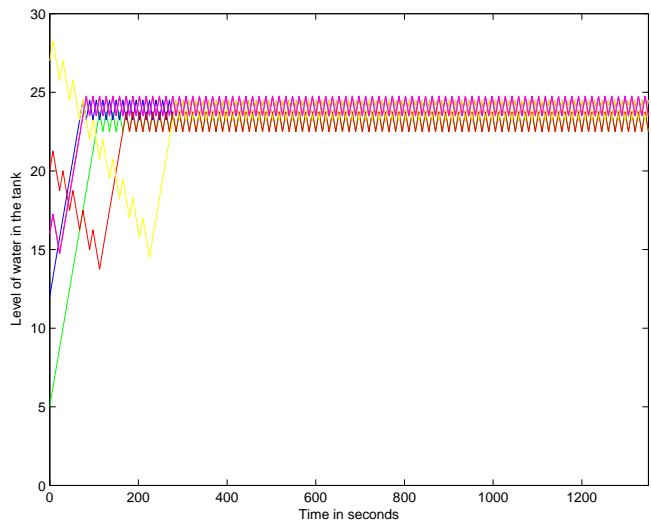}
       \caption{Water level of the tank in feedback with a DFM controller for various initial conditions.}
       \label{Fig:Example}
   \end{figure}

Consider a tank with area $A$ (sq.cm.) and height $h$ (cm), 
a binary sensor that indicates whether the water level is above or below $h/2$,
and an actuator that can pump water in or drain water out at a rate $p$ (liters/minute).
The dynamics of the sampled plant $P$, 
from which we receive a measurement $y \in \mathcal{Y}=\{\textrm{`Empty'},\textrm{`Full'}\}$
at the beginning of every sampling instant
and choose and hold a control input 
$u \in \mathcal{U}=\{\textrm{`Pump'},\textrm{`Drain'}\}$
until the next sampling instant,
is given by
\begin{displaymath}
x(t+1) = \left\{ \begin{array}{cc}
min\{h,x(t) + \frac{10^3pT}{60A} \} & \textrm{   when   } u(t)= \textrm{   `Pump'} \\
max  \{0,x(t) - \frac{10^3pT}{60A} \} & \textrm{   when    } u(t)= \textrm{   `Drain'}
\end{array} \right.
\end{displaymath}
where $T$ is the sampling interval (seconds).
Our objective is to drive and hold the water level within some desired bounds,
in the absence of exogenous input $r$.
The performance output $v$ is chosen to take the value $0$ when the water level
falls within the desired bounds and $1$ otherwise. 
The performance objective can thus be written as a gain condition (\ref{eq:ObjectiveP}),
with $\rho(r)=0$ and $\mu(v)=v$. 
Letting $A=100$, $h=30$, $p=1$, $T=7.5$, 
and choosing a desired water level between 22.5 and 25cm,
the components of the $\rho/\mu$ approximation are constructed as follows:
For $i=1$, the tank is first partitioned into $6$ equal intervals of length $h/6$,
while for each subsequent $i$ the number of elements in the partition are doubled 
(i.e. $i=2 \leftrightarrow 12$ elements, $i=3 \leftrightarrow 24$ elements,...).
The states of $\hat{M}_i$ are the elements of the partition as well as 
unions of arbitrary numbers of neighboring elements.
$\hat{M}_i$ is initialized to the state encompassing the whole tank 
(reflecting our lack of knowledge of the plant's initial state).
The transitions of  $\hat{M}_i$ are deterministic by construction, 
while its output $\hat{y}_i$ is not: 
Outputs associated with states corresponding to intervals crossing $h/2$ are 
interpreted as false predictions when computing the gain of the error system $\Delta_i$.
Error system $\Delta_i$ has input $z=u$ and output $w \in \{0,1\}$,
with $w=0$ ($w=1$) indicating a sensor output match (mismatch)
between $P$ and $\hat{M}_i$.
$\Delta_i$ is described by gain condition (\ref{eq:GainofDelta}),
where $\rho_{\Delta}(z) = \rho_{\Delta}(u)=1$ and $\mu_{\Delta}(w)=w$.
Note that the construction is similar to that proposed in \cite{JOUR:TaMeDa2011},
but with a different gain condition describing the performance objectives
as reachability specifications are considered here rather 
than exponential stability with guaranteed rate of convergence.
The performance output $\hat{v}_i$ is set to $0$ for states lying entirely within the desired bounds,
and set to $1$ otherwise.

Implementing this algorithm: For $i=1$ and $i=2$, the gain bound of $\Delta$ is 1,
and design is not successful. 
For $i=3$, the gain bound is $0$:
The approximate model thus succeeds in perfectly predicting the sensor 
output of the plant after some transient.
Moreover, control design is successful: 
Representative paths of the water level in the closed loop system, 
consisting of the plant in feedback with the controller (a DFM with 190 states)
are plotted in 
Figure \ref{Fig:Example} for various plant initial conditions.
Of course, as design is successful,
it is unecessary to construct the remaining components of the $\rho/ \mu$ approximation 
sequence for $i \geq 4$.

\section{Discussion}
\label{Sec:Comparison}

\subsection{Connections to LTI Model Reduction}

  \begin{figure*}[thpb]
       \centering
       \includegraphics[scale=0.4]{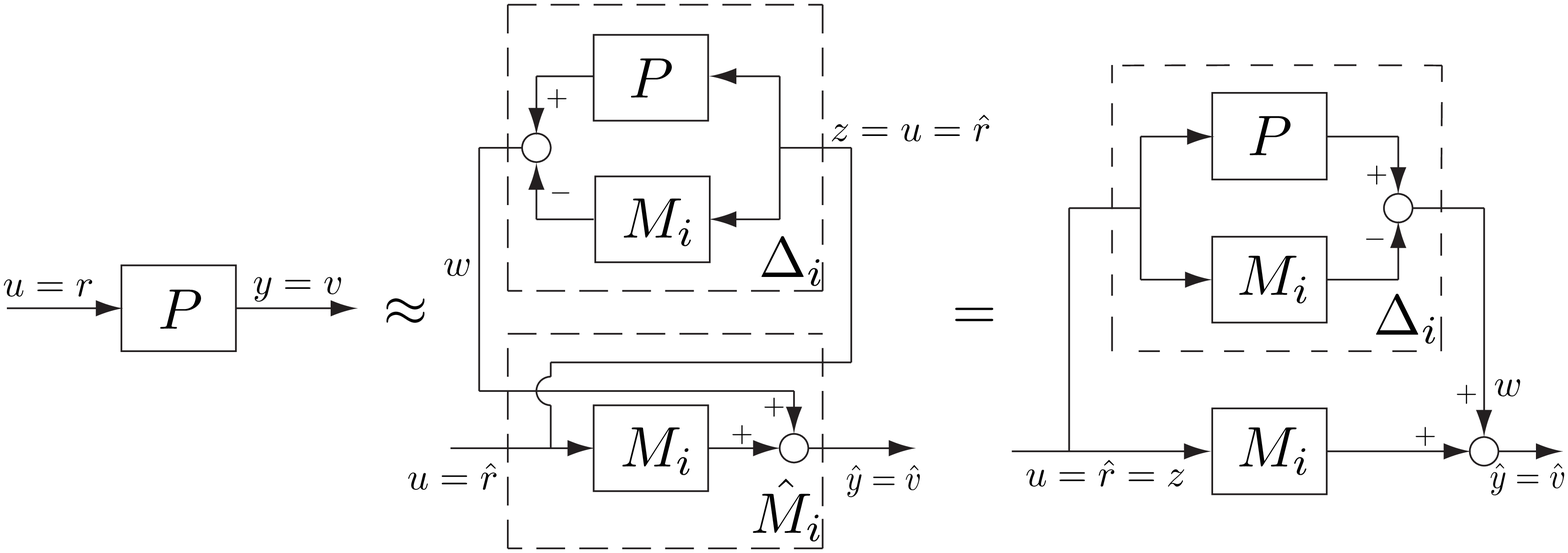}
       \caption{Definition \ref{Def:DFMApproximation} interpreted in the LTI setting.}
       \label{Fig:LTIApproximation}
   \end{figure*}

In the classical setting, 
a stable LTI plant $\tilde{P}$ of order $m$ can be considered an approximation of 
a stable LTI plant $P$ of order $n >m$ if we can recover $P$
by perturbing $\tilde{P}$ using a small stable perturbation.
The proposed notion has a similar flavor, 
with the caveat that we cannot generally hope to exactly recover the performance objective due to the 
finiteness of the input and output alphabets of a DFM.
Alternatively,
note that the notion of approximation proposed in Definition \ref{Def:DFMApproximation} 
has an interpretation in the classical setting 
(i.e. if we drop the requirements that $\hat{M}_i$ is a DFM and that $\mathcal{U}$, $\mathcal{Y}$ are finite).
Indeed, assume that $P$ is a stable LTI system of order $n$ and
each $\hat{M}_i$ is a stable LTI system of order $m_i \leq n$.
In this case, $\hat{\mathcal{R}}_i = \mathcal{R} = \mathcal{U} = \mathcal{Z}$,
$\hat{\mathcal{V}}_i = \mathcal{V} = \mathcal{Y} = \mathcal{W}$,
$\Delta_i$ is a stable LTI system given by $\Delta_i = P - M_i$ and is 
an additive perturbation of $M_i$ as shown in Figure \ref{Fig:LTIApproximation}.
Thus $\hat{P}_i=P$ and $\psi_i$ is simply the identity map.
Intuitively, $\psi$ captures the necessity, in general, to approximate the
{\it performance objective} in addition to the {\it plant} for the class of problems considered in this paper,
unless the original plant $P$ is itself a DFM. 
Moreover, additional input and output channels are needed here (for $w$ and $z$)
as signals cannot simply be added as in the LTI setting.

\subsection{Salient Features of the Proposed Notion of Approximation}

The proposed notion has three distinguishing features 
with important implications in control synthesis.
First, the design objectives are gain conditions (Definition \ref{def:GainStability}), 
and are part of the given of the problem.
Accordingly, both the plant and the performance specifications are approximated.
Second, the approximation error is characterized by the error system $\Delta$,
quantified in terms of a gain. 
Third, the relation between the original plant and its approximations is defined in terms 
of the input/output behaviors of two systems: 
$P$, and the feedback interconnection of $\hat{M}_i$ with the corresponding $\Delta_i$.
Specifically, $(\hat{M}_i,\Delta_i)$ exactly matches the control input/sensor output signal pairs of 
$P$ while satisfying additional constraints on the exogenous input/performance output signal pairs.
Consequently, 
correct-by-design control synthesis reduces in this framework to
the problem of synthesizing a controller for the DFM model so that 
the closed loop system satisfies suitable gain conditions,
a problem that can be posed and solved as a dynamic game \cite{JOUR:TaMeDa2011}. 
Moreover, this immediately yields a corresponding finite state controller for the original plant. 

\subsection{Connections to Existing Notions for Hybrid Systems}

We begin by emphasizing that all three notions of approximation enable certified-by-design controller synthesis.
In other words, if a ``sufficiently close" model is constructed and synthesis is successful,
the resulting controller guarantees that the actual closed loop system satisfies the desired specifications,
thus bypassing the need for expensive testing and verification.

Qualitative models \cite{JOUR:Lunze1994, JOUR:RaiOYo1998, JOUR:MooRai1999, JOUR:MoRaOY2002}
are similar to our proposed notion in that they characterize valid approximations in terms of {\it input/output behaviors},
and they typically address (discrete) {\it output feedback} problems.
However, they fundamentally differ in several respects:
First, in the class of nominal models considered ({\it non-deterministic} finite automata).
Second, the lack of a {\it quantitive} measure of the quality of approximation,
as approximation is simply captured by a set inclusion condition requiring
the input/output behavior of the plant to be a subset of that of its approximation. 
Third, the class of controllers (supervisory controllers) and the control synthesis procedure 
(Ramadge/Wonham framework \cite{JOUR:RamWon1987,JOUR:RamWon1989}),
which generally requires solving a dynamic programming problem for a {\it product} automaton derived from the 
approximate model and the performance specifications.

Approximate simulation/bisimulation abstractions 
\cite{JOUR:Tabuad2008, JOUR:TaAmAP2008, JOUR:GiJuPa2008, BOOK:Tabuad2009}
share one similarity with the proposed notion,
namely that they {\it quantify} the quality of approximation
through a suitably defined metric \cite{JOUR:GirPap2007}.
However, they differ from the proposed notion in two important respects:
First, they are fundamentally state-space notions that seek to relate the state trajectories
of the approximate model and the original plant,
rather than their input/output behavior.
Intuitively, an (approximate) simulation abstraction can (approximately) 
generate every possible output signal of the plant for {\it some} choice of input
{\it generally different from the corresponding input of the original system},
a detail of little consequence to verification problems 
but with ramifications on the problem of control synthesis.
Indeed, control design here is a two step procedure 
consisting of supervisory control synthesis followed by controller refinement, 
yielding a {\it hybrid} controller for the original plant \cite{BOOK:Tabuad2009}.
Second, these methods typically address {\it full state feedback} problems.

\addtolength{\textheight}{0cm}   

\section{Current \& Future Work}
\label{Sec:FutureWork}

Current research efforts are focused on developing general algorithms for constructing $\rho / \mu$ approximations. 
Preliminary efforts based on input/output partitions were reported in \cite{CONF:TarDuf2011}.
Future work will be in two additional directions:
First, exploring the use of gain conditions to encode wider classes of performance objectives. 
Specifically, we are interested in understanding to what extent temporal logic specifications, 
demonstrated to some extent in the context of the two existing notions,
can be handled by the proposed framework.
Second, quantifying the complexity
of finite memory approximations needed for a given synthesis task.
At the core of the difficulty is the state observation problem 
and the limitations imposed by the discrete output feedback.
Developments in these two directions will be instrumental in assessing the merits and drawbacks
of the proposed notion relative to the existing ones.

\section{Acknowledgments}

The author is indebted to A. Megretski for many stimulating discussions.
The author thanks M. A. Dahleh for feedback on early versions of some of the ideas presented here.
The author thanks the three anonymous reviewers and the associate editor for their helpful feedback.
This research was supported by NSF CAREER award ECCS 0954601
and AFOSR YIP award FA9550-11-1-0118.


\bibliographystyle{IEEEtranS}
\bibliography{References}

\end{document}